%% file: mislove.tex
\documentclass[11pt]{llncs}
\usepackage{latexsym,amssymb,amsmath,url}
\input macros
\renewcommand{\Pr}{\textsf{Prob}}

\begin{document}
\frontmatter
  \title{From Haar to Lebesgue via Domain Theory\\[1ex]
\emph{Revised Version}$^1$}
  \titlerunning{Haar to Lebesgue\ldots}
\author{Will Brian and Michael Mislove}
\authorrunning{Mislove}
\institute{Department of Mathematics\\Tulane University\\ New Orleans, LA 70118\\
\email{wbrian@tulane.edu, mislove@tulane.edu}}
\maketitle
\begin{abstract}
If $C \simeq 2^\Nat$ is the Cantor set realized as the infinite product of two-point groups, then a folklore result says the Cantor map from $\C$ into $[0,1]$ sends Haar measure to Lebesgue measure on the interval. In fact, $\C$ admits many distinct topological group structures. In this note, we show that the Haar measures induced by these distinct group structures all share this property. We prove this by showing that Haar measure for any group structure is the same as Haar measure induced by a related abelian group structure. Moreover, each abelian group structure on $\C$ supports a natural total order that determines a map onto the unit interval that is monotone, and hence sends intervals in $\C$ to subintervals of the unit interval. Using techniques from domain theory, we show this implies this map sends Haar measure on $\C$ to Lebesgue measure on the interval, and we then use this to contract a Borel isomorphism between any two group structures on $\C$.   \smallbreak
\end{abstract}
\setcounter{footnote}{1}
\footnotetext{The original version of this paper was published in Lectures Notes in Computer Science \textbf{8464} (2014), pp.~214-228. That version prominently claimed that any two group structures on the Cantor set have the same Haar measure, a result which is not true. This version corrects that error.}
\section{Introduction}
The discovery of the middle-third Cantor set in the late 1800s led to the first construction of a continuous map of the unit interval onto itself whose derivative is zero almost everywhere. Another remarkable -- in fact, folklore -- result about the Cantor set is that the restriction of the same map to the Cantor set sends Haar measure on the compact group $2^\Nat$ to Lebesgue measure on the interval \cite{wiki}. In this note  we generalize this result to any compact totally disconnected second countable infinite group. Any topological group structure on the Cantor set is the strict projective limit of finite groups, and conversely, the limit of a countable projective system of finite groups is a topological group on the Cantor set. In fact, any compact totally disconnected second countable group is either finite or a strict projective limit of a countable family of finite groups. 

Any locally compact group admits a unique (up to scalar factor) translation-invariant Borel measure called \emph{Haar measure}, and Haar measure is finite (and hence normalized to be a probability) measure iff the group is compact. For example, Haar measure on $(\Re,+)$ is Lebesgue measure, and Haar measure on any discrete group is counting measure. There are two main results in this paper: the first is that any topological group structure on the Cantor set has the same Haar measure as a corresponding abelian group on the Cantor set, and this implies there is a canonical map to the interval that sends Haar measure to Lebesgue measure. The second main result is that there is a Borel isomorphism between any two group structures on the Cantor set. To prove the first of these results, we first show that any strict projective system of finite groups can be replaced by a system of finite abelian groups, so that each of the replacement groups has the same cardinality as the corresponding group in the original projective system. Since the probability functor is continuous on compact Hausdorff spaces, Haar measure on the limit of a projective system of finite groups is the limit of the Haar measures on the finite groups.  Any two finite groups of the same cardinality have the same Haar measure, so this implies Haar measure on the limit of the projective system of finite abelian groups is the same as Haar measure on the limit of the original projective system. 

The advantage of a projective system of finite abelian groups is that each is a product of cyclic groups, which allows us to define a total order on each of these groups relative to which the projection maps from larger to smaller groups are monotone. This implies these total orders induce a complete total order on the limit, the Cantor set $\C$, and from this it follows that the natural map from $\C$ onto the unit interval is monotone and Lawson continuous, if we view $\C$ as a continuous lattice. Using domain theory, we then show that Haar measure on the Cantor set assigns the same length to each interval in $\C$ as Lebesgue measure assigns to the image of the interval under the map, which implies that the map sends Haar measure on the Cantor set to Lebesgue measure on the unit interval. 
\subsection{Outline of the Results}
Our focus is on the Cantor set $\C$, which can be defined abstractly as a second countable perfect Stone space, i.e., a compact Hausdorff perfect zero-dimensional space that has a countable base for its topology. Here \emph{perfect} means every point is a limit point; second countability implies $\C$ is the projective limit of a countable family of finite sets. We will study two additional structures with which $\C$ can be endowed:  
\begin{enumerate}
\item[(1)] The structure of a topological group -- the leading example is $\C\simeq 2^\Nat$, the infinite product of two-point groups, but like $2^\Nat$, any topological group structure on $\C$ can be realized as the strict projective limit of a countable system of finite groups and group homomorphisms,  and 
\item[(2)] A total order relative to which $\C$ is complete lattice.  
\end{enumerate}
Because the probability functor on compact Hausdorff spaces is continuous, viewing the Cantor set $\C$ as a compact group that is the strict projective limit of finite groups, $\C_n$ implies that Haar measure on $\C$ is the limit of the Haar measures on the $\C_n$s, where Haar measure on each $\C_n$ has the uniform distribution.  We show we can replace any topological group structure on $\C_n$ with an ``equivalent" abelian group structure, in the sense that the Haar measure is the same for both groups. As a finite abelian group, the replacement group structure is isomorphic to a finite product of cyclic groups, and we show that we can construct the replacement group $\C_n$ so that it satisfies $\C_n \simeq \bigoplus_{k\leq n} \Z_{a_k}$ is a direct product of $n$ finite cyclic groups. 

Since $\Z_k$ admits a natural total order for each $k$, this allows us to define the lexicographic order on $\C_n \simeq \bigoplus_{k\leq n} \Z_{a_k}$ for each $n$, and then the quotient mapping $\C_m \to \C_n$ is monotone for each $n\leq m$. These total orders therefore induce a complete total order on $\C$, which means $\C$ is a complete lattice in this order. Then the topology on $\C$ is the Lawson topology from the theory of continuous lattices. 

Applying Stone duality allows us to interpret each finite quotient $\C_n$ as a partition of $\C$ into subintervals, and then Haar measure on $\C_n$ assigns equal lengths to each of these intervals. Next, we show that there is a natural map from $\C$ to $[0,1]$ that is monotone and Lawson continuous. We show this assigns the same length to each subinterval of $\C$ determined by $\C_n$ as Lebesgue measure assigns to its image in $[0,1]$. 

The final piece of the puzzle relies on verifying that the length Haar measure on $\C$ assigns to each closed subinterval is the same as the length that Lebesgue measure assigns to its image in $[0,1]$.  Since both measures are continuous (i.e., they assign measure $0$ to points), and the clopen  (= closed and open) intervals in $\C$ map to the closed intervals in $[0,1]$, inner regularity implies these measures assign the same measure to open intervals, and it follows that the image of Haar measure on $\C$ under the natural map is Lebesgue measure on the interval. 

\subsection{The plan of the paper}
In the next section, we review some background material from domain theory and from the theory of compact abelian groups. Most of the latter is well-known, but we include some proofs for completeness sake. The treatment of domain theory includes a version of Stone duality. The following section constitutes the main part of the paper, where we analyze the Cantor set when it is equipped with an arbitrary  abelian topological group structure making it a topological group.  
\section{Background}\label{sec:background}
In this section we present the background material we need for our main results. 
\subsection{Domains}\label{subsec:domains}
Our results rely fundamentally on domain theory. Most of the results that we quote below can be found in \cite{abrjung} or \cite{comp}; we give specific references for those that are not found there. 

To start, a \emph{poset} is a partially ordered set. A poset is \emph{directed complete} if each of its directed subsets has a least upper bound; here a subset $S$ is \emph{directed} if each finite subset of $S$ has an upper bound in $S$. A directed complete partial order is called a \emph{dcpo}. 
The relevant maps between dcpos are the monotone maps that also preserve suprema of directed sets; these maps are usually called \emph{Scott continuous}. 

These notions can be presented from a purely topological perspective: a subset $U\subseteq P$ of a poset is \emph{Scott open} if (i) $U = \ua U \equiv \{ x\in P\mid (\exists u\in U) \ u\leq x\}$ is an upper set, and (ii) if $\sup S\in U$ implies $S\cap U\not=\emptyset$ for each directed subset $S\subseteq P$. It is routine to show that the family of Scott-open sets forms a topology on any poset; this topology satisfies $\da x \equiv \{y\in P\mid y\leq x\} = \overline{\{x\}}$ is the closure of a point, so the Scott topology is always $T_0$, but it is $T_1$ iff $P$ is a flat poset.\footnote{A space $X$ is $T_0$ if given any pair of points, there is an open set containing exactly one of the points; $X$ is $T_1$ if $\{x\}$ is a closed set for each $x\in X$.}   A mapping between dcpos is Scott continuous in the order-theoretic sense iff it is a monotone map that is continuous with respect to the Scott topologies on its domain and range. 

If $P$ is a dcpo, and $x,y\in P$, then \emph{$x$ approximates $y$} iff for every directed set $S\subseteq P$, if $y\leq \sup S$, then there is some $s\in S$ with $x\leq s$. In this case, we write $x\ll y$ and we let $\Da y = \{x\in P\mid x\ll y\}$. A \emph{basis} for a poset $P$ is a family $B\subseteq P$ satisfying $\Da y\cap B$ is directed and $y = \sup (\Da y\cap B)$ for each $y\in P$. A \emph{continuous poset} is one that has a basis, and a dcpo $P$ is a \emph{domain} if $P$ is a continuous dcpo. An element $k\in P$ is \emph{compact} if $x\ll x$, and $P$ is \emph{algebraic} if $KP = \{ k\in P\mid k\ll k\}$ forms a basis. Domains are sober spaces in the Scott topology. 

Domains also have a Hausdorff refinement of the Scott topology which will play a role in our work. The \emph{weak lower topology} on $P$ has the sets of the form if $O = P\setminus \ua F$ as a basis, where $F\subset P$ is a finite subset. The \emph{Lawson topology} on a domain $P$ is the common refinement of the Scott- and weak lower topologies on $P$. This topology has the family 
$$\{ U\setminus\!\ua F\mid U\ \text{Scott open}\ \&\ F\subseteq P\ \text{finite}\}$$
as a basis. The Lawson topology on a domain is always Hausdorff. 

A domain is \emph{coherent} if its Lawson topology is compact. We denote the closure of a subset $X\subseteq P$ of a coherent domain in the Lawson topology by $\overline{X}^\Lambda$.

\begin{example} A basic example of a domain is the unit interval; here $x\ll y$ iff $x = 0$ or $x < y$. The Scott topology on the $[0,1]$ has open sets $[0,1]$ together with $\Ua x = (x,1]$ for $x \in (0,1]$. Since  domains are closed under finite products, $[0,1]^n$ is a domain in the product order, where $x\ll y$ iff $x_i\ll y_i$ for each $i$; a basis of Scott-open sets is formed by the sets $\Ua x$ for $x\in [0,1]^n$ (this last is true in any domain). 
 
 The Lawson topology on [0,1] has basic open sets $(x,1]\setminus [y,1]$ for $x < y$ -- i.e., sets of the form $(x,y)$ for $x < y$, which is the usual topology. Then, the Lawson topology on $[0,1]^n$ is the  product topology from the usual topology on $[0,1]$. This shows $[0,1]$ is a coherent domain. 
 
 Since $[0,1]$ has a least element, the same results apply for any power of $[0,1]$, where $x\ll y$ in $[0,1]^J$ iff $x_j = 0$ for almost all $j\in J$, and $x_j\ll y_j$ for all $j\in J$. Thus, every power of $[0,1]$ is a coherent domain. 

 Similarly, the middle-third Cantor set $\C\subseteq [0,1]$ is a domain in the order it inherits from $[0,1]$. But while $K[0,1] = \{0\}$, the compact elements of $\C$ consist of the least upper bounds of the open intervals that are deleted from $[0,1]$ to form $\C$ -- ${2\over 3}, {2\over 9}, {8\over 9}$,\ldots. Thus, $y = \sup K\C\cap \da y$ for each $y\in \C$, so $\C$ is an \emph{algebraic} domain, in fact a complete algebraic lattice. 

A more interesting example of a coherent domain is $\Pr(D)$, the family  of probability measures on a coherent domain $D$, where $\mu\leq \nu$ iff $\mu(U)\leq \nu(U)$ for every Scott-open subset $U\subseteq D$. For example,  $\Pr([0,1])$ is a coherent domain. In fact, the category \textsf{COH} of coherent domains and Scott continuous maps is closed under the application of the functor \textsf{Prob}~\cite{jungtix}.  
 \end{example}

\subsubsection{Embedding-projection Pairs}
One of the features of domain theory is its ability to provide solutions to \emph{domain equations} -- these are abstract domains that satisfy structural requirements, most often ones needed in defining models for programming language constructs. Of course, the most famous domain equation is $D\simeq [D\to D]$, which can be solved in any of the number of Cartesian closed categories of domains. We don't need anything so sophisticated, but we can use the basic approach to solving domain equations to realize Stone spaces as algebraic lattices. 

\begin{definition}\label{def:e-p-pair}
Let $P$ and $Q$ be posets. An \emph{embedding--projection pair} between $P$ and $Q$ is a pair of monotone mappings $e\colon Q\to P$ and $p\colon P\to Q$ satisfying $p\circ e = 1_Q$ and $p\circ e\leq 1_P$, where the order on functions is pointwise. 
\end{definition}

The main result we need is the following:

\begin{theorem}\label{thm:e-p-pair}
Let $(P_i,e_{i,j},p_{i.j})_{i\leq j\in I}$ be an indexed family of domains $P_i$ and Scott-continuous e--p pairs $e_{i,j}\colon P_i\to P_j$, $p_{i,j}\colon P_j\to P_i$ for $i\leq j$. Then $P = \{(x_i)_{i\in I}\mid p_{i,j}(x_j) = x_i\}$ is a domain, and the projection maps $\pi_i\colon P\to P_i$ together with the mappings $e_i\colon P_i\to P$ by $(e_i(x))_j = \begin{cases} p_{i,j}(x) & \text{if}\ i\leq j\cr
							e_{i,j}(x_j) & \text{if}\ j\leq i\cr
							\end{cases}$ form Scott-continuous e--p pairs. 
							Moreover, if each $P_i$ is algebraic, then so is $P$, and $KP = \bigcup_i e_i(KP_i)$.
\end{theorem}

\subsection{The \textsf{Prob} monad on \textsf{Comp}}
It is well known that the family of probability measures on a compact Hausdorff space is the object level of a functor which defines a monad on \textsf{Comp}, the category of compact Hausdorff spaces and continuous maps. As outlined in \cite{hofmis}, this monad gives rise to several related monads:
\begin{itemize}
\item On \textsf{Comp}, it associates to a compact Hausdorff space $X$ the free \emph{barycentric algebra} over $X$, the name deriving from the counit $\epsilon\colon \Pr(S) \to S$ which assigns to each measure $\mu$ on a probabilistic algebra $S$ its barycenter $\epsilon(\mu)$. 
\item On the category \textsf{CompMon} of compact monoids and continuous monoid homomorphisms, \textsf{Prob} gives rise to a monad that assigns to a compact monoid $S$ the free compact affine monoid over $S$.
\item On the category \textsf{CompGrp} of compact groups and continuous homomorphisms, \textsf{Prob} assigns to a compact group $G$ the free compact affine monoid over $G$; in this case the right adjoint sends a compact affine monoid to its group of units, as opposed to the inclusion functor, which is the right adjoint in the first two cases.
\end{itemize}
If we let \textsf{SProb}$(X)$ denote the family of subprobability measures on a compact Hausdorff space $X$, then it's routine to show that \textsf{SProb} defines monads in each of the cases just described, where the only change is that the objects now have a $0$ (i.e., they are affine structures with $0$-element, allowing one to define scalar multiples $r\cdot x$ for $r\in [0,1]$ and $x\in\textsf{SProb}(X)$, as well as affine combinations).

There is a further result we need about \textsf{Prob} which relates to its role as an endofunctor on \textsf{Comp} and its subcategories. The following result is due to Fedorchuk:
\begin{theorem}[Fedorchuk~\cite{fedorchuk}]\label{thm:fedorchuk} The functor \textsf{Prob}$\colon \textsf{Comp}\to \textsf{Comp}$ is normal; in particular, \textsf{Prob} preserves inverse limits.
\end{theorem}

\subsection{Stone duality}\label{subsec:stone}
In modern parlance, Marshall Stone's seminal result states that the category of Stone spaces -- compact Hausdorff totally disconnected spaces -- and continuous maps is dually equivalent to the category of Boolean algebras and Boolean algebra maps. The dual equivalence sends a Stone space to the Boolean algebra of its compact-open subsets; dually, a Boolean algebra is sent to the set of prime ideals, endowed with the hull-kernel topology. This dual equivalence was used to great effect by Abramsky~\cite{abr-dominlogform} where he showed how to extract a logic from a domain constructed using Moggi's monadic approach, so that the logic was tailor-made for the domain used to build it. 

Our approach to Stone duality is somewhat unconventional, but one that also has been utilized in recent work by Gehrke~\cite{gehrke1,gehrke2}. The idea is to realize a Stone space as a projective limit of finite spaces, a result which follows from Stone duality, as we now demonstrate.

\begin{theorem} [Stone Duality]\label{thm:stone} Each Stone space $X$ can be represented as a projective limit $X \simeq \varprojlim_{\alpha\in A} X_\alpha$, where $X_\alpha$ is a finite space. In fact, each $X_\alpha$ is a partition of $X$ into a finite cover by clopen subsets, and the projection $X \twoheadrightarrow X_\alpha$ maps each point of $X$ to the element of $X_\alpha$ containing it. 
\end{theorem}
\begin{proof}
If $X$ is a Stone space, then $\B(X)$, the family of compact-open subsets of $X$ is a Boolean algebra. Clearly $\B(X)\simeq \varinjlim_{\alpha\in A} \B_\alpha$ is the injective limit of its family $\{\B_\alpha\mid \alpha \in A\}$ of finite Boolean subalgebras. For a given $\alpha\in A$, we let $X_\alpha$ denote the finite set of atoms of $\B_\alpha$. Then $\B_\alpha\hookrightarrow \B(X)$ implies $\B_\alpha$ is a family of clopen subsets of $X$, and the set of atoms of $\B_\alpha$ are pairwise disjoint, and their sup -- i.e., union -- is all of $X$, so $X_\alpha$ forms a partition of $X$ into clopen subsets, Thus there is a continuous surmorphism $X\twoheadrightarrow X_\alpha$ sending each element of $X$ to the unique atom in $X_\alpha$ containing it. The family $\{\B_\alpha\mid \alpha\in A\}$ is an injective system, since given $\B_\alpha$ and $\B_\beta$, the Boolean subalgebra they generate is again finite. Dually the family $\{ X_\alpha\mid \alpha\in A\}$ is a projective system, and since $\B(X)\simeq \varinjlim_{\alpha\in A} \B_\alpha$, it follows that $X\simeq \varprojlim_{\alpha\in A} X_\alpha$.
\end{proof}

We note that a corollary of this result says that it is enough to have a basis for the family of finite Boolean subalgebras of $\B(X)$ in order to realize $X$ as a projective limit of finite spaces, where by a \emph{basis}, we mean a directed family whose union generates all of $\B(X)$.

\subsection{Compact Groups}
We now recall some results about compact topological groups. We include proofs of some results that are well-known in the interest of completeness. A standard reference for group theory is \cite{Rot}, and an excellent reference for the theory of compact groups is \cite{hofmor} 

To begin, a \emph{topological group} is a $T_1$-topological space $G$ that is also a group, and for which the multiplication $\cdot \colon G\times G\to G$ and inversion $x\mapsto x^{-1}\colon G\to G$ mappings are continuous. A basic result is that all topological groups are Hausdorff spaces. A \emph{compact group} is a topological group whose topology is compact. 

We are interested in group structures on the Cantor set, which can be characterized as a metrizable perfect Stone space. That is, a \emph{Cantor set} is a compact Hausdorff zero-dimensional space that has a countable base for its topology, and in which every point is a limit point. It is well-known that any Cantor set has a base of clopen subsets. We prove a stronger result for groups on the Cantor set.

\begin{proposition} If $G$ is a compact group whose underlying space it zero dimensional, then $G$ admits a neighborhood base of the identity consisting of clopen normal subgroups. 
\end{proposition} 
\begin{proof}
We start with a basis $\O$ of clopen neighborhoods of the identity, which exists in any Stone space. Since inversion is a homeomorphism (being its own inverse), each $O\in\O$ satisfies $O^{-1}\in \O$, so $O\cap O^{-1}\in \O$, which implies it is no loss of generality to assume that $O = O^{-1}$ for each $O\in \O$. 

Now, since multiplication is continuous and $O$ is both compact and open, $O = e\cdot O \subseteq O$ implies there is a $U\in \O$ with $U\cdot O\subseteq O$. But then $U\subseteq O$, and so $U^2\subseteq O$, and by induction, $U^n\subseteq O$ for each $n > 0$. Since $U$ is symmetric, this implies the subgroup $H_U$ that $U$ generates is a subset of $O$. And since $U$ is open, so is $H_U$  (which also implies $H_U$ is closed). 

For the claim about normal subgroups, we first recall that the family of conjugates $\H = \{xHx^{-1}\mid x\in G\}$ of a closed subgroup $H < G$ is closed in the space of closed subsets of $G$ endowed with the Vietoris topology, which is compact since $G$ is compact. Moreover, $G$ acts continuously on $\H$ by $(x,H) \mapsto xHx^{-1}\colon G\times \H\to \H$. The kernel $K = \{ x\in G\mid xHx^{-1} = H\}$ of this action is then a normal subgroup of $G$, and if $H$ is clopen, then $K$ is clopen as well. But since $G$ acts transitively on this family of conjugates, it follows that $|G/K | = |\H|$. Since $K$ is open and $G$ is compact, $G/K$ is finite, and so there are only finitely many cosets $xHx^{-1}$. Then their intersection $\bigcap_{x\in G} xHx^{-1}\subseteq H$ is a clopen normal subgroup of $G$ inside $H$. Since $G$ has a basis of clopen subgroups $H$ around $e$ by the first part, and we can refine each of these with a clopen normal subgroup by taking $\bigcap_{x\in G} xHx^{-1}$, it follows that $G$ has a basis of clopen normal subgroups around $e$.
\end{proof}

\begin{corollary}\label{cor:invlim}
Any compact zero-dimensional group is the strict projective limit\,\footnote{A projective system is \emph{strict} if the projection maps all are surjections.} of finite groups. 
\end{corollary}
\begin{proof}
If $G$ is compact and zero-dimensional, then $e$ has a basis $\N$ of clopen normal subgroups by the Proposition. If $N\in\N$, then since $G$ is compact, $G/N$ also is compact and the quotient map $\pi_N\colon G \to G/N$ is open. But $N\in \N$ is open, so $G/N$ is discrete, which implies there are only finitely many cosets in $G/N$, i.e., $G/N$ is finite. The family $\N$ is directed by reverse set inclusion, and for $M\subseteq N\in\N$, we let $\pi_{N,M}\colon G/M\to G/N$ be the natural projection. Then the family $(G/N,\pi_{N,M})_{M\leq N\in\N}$ forms a strict projective system of finite groups which satisfies $G\simeq \varprojlim_\N G/N$. 
\end{proof}

\begin{remark}
We also note that since any topological group is homogeneous, a topological group must satisfy the property that either every point is a limit point, or else the group is discrete. Thus, the underlying space of a compact group is either perfect or the group is finite. In particular, a topological group on a Stone space forces the space to be finite or perfect. By a \emph{Cantor group}, we mean a topological group structure on a Cantor set (which we also assume is metrizable).
\end{remark}

\subsection{Haar Measure on Cantor Groups}

\begin{definition} A Borel measure $\mu$ on a topological group $G$ is \emph{left translation invariant} if $\mu(xA) = \mu(A)$ for all $x\in G$ and all measurable sets $A\subseteq G$. 
\end{definition}
A fundamental result of topological group theory is that each locally compact group admits a left translation invariant Borel measure which is unique up to scalar constant; i.e., if $\mu$ and $\nu$ are left translation invariant measures on the locally compact group $G$, then there is a constant $c> 0$ such that $\mu(A) = c\cdot \nu(A)$ for every measurable set $A$. Any such measure is called a \emph{Haar measure}. If $G$ is compact, the measure $\mu$ is assumed to satisfy $\mu(G) = 1$, which means this measure is unique. Notice in particular that Haar measure on any discrete group is counting measure, and on a finite group, it is normalized counting measure. 

We now establish an important result we need for the main result of this section.

\begin{proposition}\label{prop:haarmap}
Let $G$ and $H$ be compact groups and let $\phi\colon G\to H$ be a continuous surmorphism. Then $\phi(\mu_G) = \mu_H$, where $\mu_G$ and $\mu_H$ are Haar measure on $G$ and $H$, respectively. 
\end{proposition}
\begin{proof}
Let $K = \ker \phi$, and let  $A\subseteq G/K$ be measurable and $x\in G/K$. Since $\phi$ is a surmorphism, there is $x_0\in G$ with $\phi(x_0) = x$. Then 
\begin{eqnarray*}
\pi_K(\mu_G)(xA) &=& \mu_G(\phi^{-1}(xA)) = \mu_G(\phi^{-1}(x)\cdot \phi^{-1}(A)) = \mu_G(x_0 K\cdot \phi^{-1}(A))\\ & {\buildrel *\over =}& \mu_G(x_0\phi^{-1}(A)) = \mu_G(\phi^{-1}(A)) = \pi_K(\mu_G)(A),
\end{eqnarray*}
where ${\buildrel *\over =}$ follows from the normality of $K$ and the fact that $\phi^{-1}(A)$ is saturated with respect to $K$, and the next equality follows because $\mu_G$ is Haar measure on $G$. Thus $\phi(\mu_G)$ is a Haar measure on $H$. The result then follows by the uniqueness of Haar measure on a compact group.  
\end{proof}

The main result of this section is the following:

\begin{theorem}\label{thm:sameHaar}
If $G$ is a topological group whose underlying space is a Cantor set, then there is an abelian topological group structure on $G$ that has the same Haar measure as the original topological group structure. 
\end{theorem}

\begin{proof}
Since $G$ is a Cantor set, Corollary~\ref{cor:invlim} implies $G\simeq \varprojlim_k G/N_k$ of a countable chain of finite groups, where $k\leq k'$ implies $N_{k'}\subseteq N_k$. For each $k$, we define groups $G_k$ as follows:
\begin{enumerate}
\item $G_1 = \Z_{n_1}$, where $n_1 = |G/N_1|$, and 
\item for $k > 1$, $G_k = G_{k-1}\oplus \Z_{n_k}$, where $n_k = |\ker \pi_{G/N_k,G/N_{k-1}}|$.
\end{enumerate}
In short, $G_k = \oplus_{l\leq k} \Z_{n_l}$, where $n_1 = |G/N_1|$, and $n_k  =  |\ker \pi_{G/N_k,G/N_{k-1}}|$ for $k > 1$. Thus, $G_k$ is a direct product of cyclic groups, and $|G_k| = |G/N_k|$ for each $k$. Since $G_k$ and $G/N_k$ are both finite, this last implies Haar measure on $G/N_k$ is the same as Haar measure on $G_k$ for each $k$. 

Clearly there is a canonical projection $\pi_{k,k'}\colon G_k \to G_{k'}$ whenever $k' \leq k$. So we have a second strict projective system $(G_k,\pi_{k,k'})_{k'\leq k}$, and since $G/N_k\simeq G_k$ for each $k$ \emph{qua} topological spaces, it follows that $G\simeq \varprojlim_{k'\leq k} (G_k, \pi_{k',k})$, again \emph{qua} topological spaces. 

Next, Theorem~\ref{thm:fedorchuk} implies that the limit of the sequence $\{\mu_{G_k}\}_k$ is a Borel measure $\mu$ on $G\simeq \varprojlim_{k'\leq k} (G_k, \pi_{k',k})$ whose image under the quotient map $G \to G_k$ is $\mu_k$, Haar measure on $G_k$. But if $G_A$ denotes the limit of the projective system $(G_k,\pi_{k',k})_{k'\leq k}$ \emph{qua} compact abelian groups, then Proposition~\ref{prop:haarmap}  implies Haar measure on $G_A$ also has image $\mu_{G_k}$ under the quotient map $G_A\to G_k$. Since limits are unique, this implies $\mu_{G_A} = \mu$. 

Now the Haar measures $\mu_{G/N_k} = \mu_k$ on $G/N_k$ and on $G_k$ are equal by design, and Proposition~\ref{prop:haarmap} implies Haar measure $\mu_G$ on $G$ with its original  compact group structure maps to $\mu_{G/N_k}$ under the quotient map $G \to G/N_k$. Again limits are unique, so we conclude that $\mu_G = \mu_{G_A}$, the Haar measure induced on $G_A\simeq \varprojlim_{k'\leq k} (G_k, \pi_{k',k})$ \emph{qua} compact abelian group.
\end{proof}

\begin{remark}\
We note that the same result holds for general (ie., nonmetrizable) compact group structures on Stone spaces. The only thing that changes is that the group may require a directed family of finite quotients that may be uncountable. 
\end{remark}

\section{Defining an Bialgebraic Lattice Structure on $\C$}
According to the proof of Theorem~\ref{thm:sameHaar} we can assume we are given a Cantor group $\C \simeq \varprojlim (\C_n,\pi_{m,n})_{n\leq m}$ where each $\C_n = \oplus_{i\leq n} \Z_{n_i}$ is a product of $n$ cyclic groups, and the mapping $\pi_{m,n}\colon \C_m\to \C_n$ is the projection map onto the first $n$ factors of $\C_m$ for $n\leq m$. In particular, this representation relies on a fixed sequence of finite cyclic groups $\{\Z_{n_i}\mid i > 0\}$ satisfying $\C_n = \oplus_{i\leq n} \Z_{n_i}$, and without loss of generality, we can assume that $n_i > 1$ for each $i$ -- this follows from the fact that $\C$ is a perfect (hence uncountable) Stone space and each quotient group $\C_n$ is finite. 

\begin{theorem}\label{thm:totord}
$\C$ admits a total order relative to which it is a complete bialgebraic lattice\footnote{A lattice $L$ is \emph{bialgebraic} if $L$ and $L^{op}$ are both algebraic lattices.} endowed with the Lawson topology. 
\end{theorem}

\begin{proof}
We first note that we can define a total order on $\C_n = \oplus_{n_i\leq n} \Z_{n_i}$ to be the lexicographic order, where we endow $\Z_{n_i}$ with its total order from $\Nat$. 

Next, the projection mapping $\pi_{m,n}\colon \C_n \to \C_m$ is monotone and clearly Scott continuous, for $n\leq m$, and we can define embeddings $\iota_{m,n}\colon \C_m\to \C_n$ by $\iota_{m,n}(x)_i = \begin{cases} x_i & \text{if}\ j\leq m\cr
0 & \text{if}\ m < j\cr
\end{cases}$, and clearly $\iota_{m.n}$ is monotone and Scott continuous. Moreover, it is clear that $\pi_{m,n}\circ \iota_{m,n} = 1_{\C_n}$ and $\iota_{m,n}\circ \pi_{m,n} \leq 1_{\C_m}$ for $n\leq m$. 

So,  we have a  system $((\C_n,\leq_n), \iota_{m,n}, \pi_{m,n})_{n\leq m}$ of e--p pairs in the category of algebraic lattices and Scott-continuous maps. By Theorem~\ref{thm:e-p-pair}, $\varprojlim ((\C_n),\leq_n), \pi_{n,m})_{n\leq m}$ is an algebraic lattice whose compact elements are the union of the images of the $\C_n$s under the natural embeddings, $\iota_n(x)_j = \begin{cases} x_j & \text{if}\ j\leq n\cr 0 & \text{if}\ n < j\cr\end{cases}$. But this is the same family of finite sets and projection maps that define the original projective system, which implies $\C$ has a total order relative to which it is an algebraic lattice. 

To see that $\C^{op}$ also is algebraic, we note that since $\C$ is totally ordered and complete, each $x\in K\C$ has a corresponding $x' = \sup (\da x\setminus\{x\}) \in K\C^{op}$. If $y \not\in K\C$ and $y < z\in \C$, then since $\C$ is algebraic, $z = \sup (\da z\cap K\C)$, so there is some $x\in K\C$ with $y < x \leq z$. But then $y\leq x'\in K\C^{op}$. It follows that $y = \inf (\ua y\cap K\C^{op})$ for $y \in \C$, so $\C^{op}$ also is algebraic. 

Finally, the Lawson topology on an algebraic lattice is compact and Hausdorff, and it is refined by the original topology on $\C$, so the two topologies agree.
\end{proof}

\begin{remark}
We note that $K\C = \bigcup_n \iota_n(\C_n)$, and the mappings $\iota_n\colon \C_n\to K\C$ are injections, so we often elide the injections $\iota_n$ and simply regard $\C_n$ as a subset of $\C$. Note as well that $\iota_n$ is a group homomorphism for each $n$, so this identification applies both order-theoretically and group theoretically.
\end{remark}
Theorem~\ref{thm:totord} allows us to define the natural map $\phi\colon \C\to [0,1]$: For each $n$, $\C_n = \oplus_{i\leq n} \Z_{n_i}$, endowed with the lexicographic order. For $x\in \C_n$, we define $\phi_n(x) = \sum_{i\leq n} {x_i\over n_1\cdot n_2\cdots n_i}$.\footnote{An intuitive way to understand $\phi_n$ for each $n$ is that $\Z_{n_1}$ divides the interval into $n_1$ subintervals, $\Z_{n_2}$ divides each of those into $n_2$ subintervals, and so on. So  $\phi_n$ maps the elements of $\C_n$ to those those rationals in $[0,1]$ that can be expressed precisely in an expansion using $n_1, n_2,\ldots$ as  successive denominators.} Then $\phi_n$ is monotone, and $n\leq m$ implies $\phi_m\circ \iota_{m,n} = \phi_n$. Thus we have a monotone mapping $\phi\colon K\C\to [0,1]$. The fundamental theorem of domain theory implies $\phi$ admits a Scott-continuous extension $\widehat{\phi}\colon \C\to [0,1]$. 

In fact, note that $\phi\colon K\C\to [0,1]$ is stictly monotone: if $x < y$, then $\phi(x) < \phi(y)$. This implies $\phi$ is one-to-one on $K\C$, and clearly its image is dense in $[0,1]$. Now, for any $s\in (0,1]$, if $x\in \C$ satisfies $\widehat{\phi}(x) < s$, then $s - \widehat{\phi}(x) > 0$, so we can choose $n > 0$ large enough so there are $x_n, y_n \in \C_n \subseteq K\C$ satisfying $x_n \leq x < y_n$ and $\phi_n(y_n) < s$. Hence $\C\setminus \widehat{\phi}^{-1}([s,1])$ is weak-lower open in $\C$, from which it follows that $\widehat{\phi}^{-1}([s,1])$ is weak-lower closed, which is to say $ \widehat{\phi}^{-1}([s,1]) = \ua z$ for some $z\in \C$. But this implies that $\widehat{\phi}$ is Lawson continuous. Since $\C$ is compact in the Lawson topology, this implies $\widehat{\phi}(\C) = [0,1]$. 

Moreover, since $[0,1]$ is connected and $\widehat{\phi}$ is monotone, it follows that $\widehat{\phi}(x') = \widehat{\phi}(x)$ for each $x\in K\C$. We summarize this discussion as

\begin{corollary}\label{cor:lawcont}
The mapping $\widehat{\phi}\colon K\C\to [0,1]$ by $\widehat{\phi}(x) = \sum_{i\leq n} {x_i\over n_1\cdot n_2\cdots n_i}$ is strictly monotone (hence injective), and it has a Lawson-continuous, monotone and surjective extension defined by $\widehat{\phi}(x) = \sup \widehat{\phi}(\da x\cap K\C)$.
Moreover, for each $x\in K\C$, $\widehat{\phi}(x') = \widehat{\phi}(x)$, where $x' = \sup (\da x\setminus\{x\})\in KC^{op}$. 
\end{corollary}

\section{Mapping Haar Measure to Lebesgue Measure}
We now come to the main result of the paper. Our goal is to show that there is a natural map from any Cantor group onto the unit interval that sends Haar measure to Lebesgue measure. According to Theorem~\ref{thm:sameHaar}, any compact group structure on a Cantor set has the same Haar measure as a group structure realized as the strict projective limit of a sequence of finite abelian groups, and Theorem~\ref{thm:totord} and Corollary~\ref{cor:lawcont} show there is a Lawson continuous monotone mapping of $\C$ onto the unit interval for such a group structure. We now show that this map sends Haar measure on $\C$ as an abelian group to Lebesgue measure. 

Recall that the abelian group structure satisfies $\C = \varprojlim_{n >0} (\bigoplus_{i\leq n} \Z_{k_i}, \pi_{m,n})$, where $k_i > 1$ for each $i$,  and $\pi_{m,n}\colon \C_m \to \C_n$ is the projection on the first $n$ factors, for $n\leq m$. Theorem~\ref{thm:totord} says $\C$ has a total order relative to which it is a complete bialgebraic lattice, and it is this order structure we exploit in our proof. 

\begin{remark}
 Recall that $b\in K\C$ implies $b' = \sup(\da b\setminus\{b\})\in KC^{op}$ and $\widehat{\phi}(b) = \widehat{\phi}(b')$. We need this fact because the clopen intervals in $\C$ all have the form $[a,b']$ for some $a\leq b\in K\C$. Indeed, according to Stone duality (Theorem~\ref{thm:stone}), in a representation of a Stone space as a strict projective limit of finite spaces, each finite quotient space corresponds to a partition of the space into clopen sets. If the Stone space is totally ordered and the representation is via monotone maps, then the elements of each partition are clopen intervals. In particular, if $\pi_n\colon\C\to \C_n$ is a projection map, then $\pi_n^{-1}(x) = [a,b']$ for some $a,b\in K\C$, for each $x\in \C_n$
 
 Throughout the following, we let $\mu_\C$ denote Haar measure on $\C$, and let $\lambda$ denote Lebesgue measure on $[0,1]$. 
 \end{remark}

\begin{proposition}\label{prop:Gnmap}
If $a\leq b\in \C_n$, then $\lambda(\widehat{\phi}([a,b'])) = \mu_{\C_n}([a,b]_{\C_n})$. 
\end{proposition}
\begin{proof} On one hand, $\lambda(\widehat{\phi}([a,b'])) = \widehat{\phi}(b') - \widehat{\phi}(a) =  \widehat{\phi}(b) - \widehat{\phi}(a)$. On the other, $\mu_{\C_n}([a,b]_{\C_n}) = {|[a,b]_{\C_n}|\over |\C_n|}$ since $\C_n$ is finite. Now $\C_n = \bigoplus_{i\leq n} \Z_{k_i}$ in the lexicographic order, and we show these two values agree by induction on $n$. Indeed, since $a \leq b \in \C_n$, we have $a = (a_1,\ldots, a_i)$ and $b = (b_1,\ldots, b_j)$ for some $i,j\leq n$, and then $\widehat{\phi}(a) = \sum_{l\leq i} {a_l\over k_1\cdots k_l}$ and $\widehat{\phi}(b) = \sum_{l\leq j} {b_l\over k_1\cdots b_l}$. By padding $a$ and $b$ with $0$s, we can assume $i=j=n$. Then
\begin{eqnarray*} 
\lambda(\widehat{\phi}([a,b'])) & = &\widehat{\phi}(b') - \widehat{\phi}(a)  = \widehat{\phi}(b) - \widehat{\phi}(a)\\ &=& \sum_{l\leq n} {b_l\over k_1\cdots b_l} - 
\sum_{l\leq n} {a_l\over k_1\cdots k_j}\\
&=& \left(\sum_{l\leq n-1} {b_l\over k_1\cdots b_l} - 
\sum_{l\leq n-1} {a_l\over k_1\cdots k_i}\right)\\ && \qquad + \left(\left|{b_n\over k_1\cdots k_n } - {a_n\over k_1\cdots k_n}\right|\right)\\
&{\buildrel \dag\over =} & {|[a^*,b^*]_{\C_{n-1}}|\over|\C_{n-1}|}  + \left(\left|{b_n\over k_1\cdots k_n } - {a_n\over k_1\cdots k_n}\right|\right)\\
& = & {|[a,b]_{\C_{n}}|\over|\C_{n}|} = \mu_\C([a,b]_{\C_n}), 
\end{eqnarray*}
where $a^* = (a_1,\ldots, a_{n-1})\leq b^*=(b_1,\ldots, b_{n-1})\in \C_{n-1}$ so that ${\buildrel \dag\over =}$ follows by induction.
\end{proof}

\begin{theorem}\label{thm:openmap}
Let $\O([0,1])$ denote the family of open subsets of $[0,1]$. Then $\lambda\colon \O([0,1])\to [0,1]$ and $\mu_\C\circ \widehat{\phi}^{-1}\colon \O([0,1])\to [0,1]$ are the same mapping. 
\end{theorem}

\begin{proof}
Let $U\in\O([0,1])$ be an open set. 
Since $\widehat{\phi}$ is Lawson continuous, $\widehat{\phi}^{-1}(U)$ is open in $\C$, and since $\C$ is a Stone space, it follows that $\widehat{\phi}^{-1}(U) = \bigcup \{ K\mid K\subseteq \widehat{\phi}^{-1}(U)\ \text{clopen}\}$. Now, $\widehat{\phi}$ is a continuous surjection, so $\widehat{\phi}(K)$ is compact and 
\begin{eqnarray*}
U = \widehat{\phi}(\widehat{\phi}^{-1}(U)) &= & \widehat{\phi}\left(\bigcup \{ K\mid K\subseteq \widehat{\phi}^{-1}(U)\ \text{clopen}\}\right)\\
& =& \bigcup \{ \widehat{\phi}(K)\mid K\subseteq \widehat{\phi}^{-1}(U)\ \text{clopen}\}.
\end{eqnarray*}
Next, any clopen $K\subseteq \C$ is compact, and because $K\C$ is dense, $(\exists n > 0)(\exists a_i< b_i\in \C_i)\, K = \bigcup_{i\leq n} [a_i,b_i']$.  Moreover, we can assume $[a_i,b_i']\cap [a_j,b_j'] = \emptyset$ for $i\not= j$. Then 
$$\mu_\C(K) = \sum_i \mu_\C([a_i,b_i']) = \sum_i \lambda(\widehat{\phi}([a_i,b_i'])),$$ 
the last equality following from Proposition~\ref{prop:Gnmap}. Since the intervals $[a_i,b_i']$ are pairwise disjoint, if $\widehat{\phi}([a_i,b_i'])\cap \widehat{\phi}([a_j,b_j'])\not= \emptyset$ then either $b_i' = a_j'$ or $b_j' = a_i'$. In either case, $\widehat{\phi}([a_i,b_i'])\cap \widehat{\phi}([a_j,b_j'])$ is singleton, and then since $\lambda$ is continuous,
\begin{equation}\label{eqn:one}
\mu_\C(K) = \sum_i \lambda(\widehat{\phi}([a_i,b_i'])) = \lambda(\bigcup_i \widehat{\phi}([a_i,b_i'])) = \lambda(K).
\end{equation}
Finally, since $\mu_\C$ and $\lambda$ are both inner regular, we have
\begin{eqnarray*}
\lambda(U) &=& \lambda(\widehat{\phi}(\widehat{\phi}^{-1}(U)))\\ &=&  \lambda\left(\bigcup \{ \widehat{\phi}(K)\mid K\subseteq \widehat{\phi}^{-1}(U)\ \text{clopen}\}\right)\\
&{\buildrel \dag\over =}& \bigcup\{\lambda(\widehat{\phi}(K))\mid K\subseteq \widehat{\phi}^{-1}(U)\ \text{clopen}\}\\
&{\buildrel \ddag\over =}& \bigcup\{\mu_\C(K)\mid K\subseteq \widehat{\phi}^{-1}(U)\ \text{clopen}\}\\
&{\buildrel \#\over =}& \mu_\C\left(\bigcup\{K\mid K\subseteq \widehat{\phi}^{-1}(U)\ \text{clopen}\}\right)\\ &=& \mu_\C(\widehat{\phi}^{-1}(U)). 
\end{eqnarray*}
where $\buildrel \dag\over =$ follows by the inner regularity of $\lambda$,\footnote{It is straightforward to argue that \emph{any} compact set $C\subseteq U$ is contained in $\bigcup \{ \widehat{\phi}(K)\mid K\subseteq \widehat{\phi}^{-1}(U)\ \text{clopen}\}$.} $\buildrel \ddag\over =$ follows from Equation~\ref{eqn:one}, and $\buildrel \#\over =$ follows from the inner regularity of $\mu_\C$. 
\end{proof}

\begin{corollary}\label{cor:haarmap}
If we endow $\C$ with the structure of topological group with  Haar measure $\mu_\C$, then there is a continuous mapping $\widehat{\phi}\colon \C\to [0,1]$ satisfying
$\widehat{\phi}(\mu_\C) = \lambda$.
\end{corollary}
\begin{proof}
If $A\subseteq [0,1]$ is Borel measurable, then $\widehat{\phi}(\mu_\C)(A) = \mu_\C(\widehat{\phi}^{-1}(A)$. We have shown $\widehat{\phi}(\mu_\C)(A) = \lambda(A)$ in case $A$ is open. But since the open sets generate the Borel $\sigma$-algebra the result follows.
\end{proof}

\begin{theorem}\label{thm:allsame}
Let $\C_1$ and $\C_2$ be Cantor sets with topological group structures. Then there is aBorel isomorphism between $\C_1$ and $\C_2$. 
\end{theorem}
\begin{proof}
By Theorem~\ref{thm:sameHaar}, we can assume that the group structures on $\C_1$ and $\C_2$ are both abelian, and then Theorem~\ref{thm:totord} and Corollary~\ref{cor:lawcont} show there are Lawson-continuous monotone mappings of $\widehat{\phi}_1\colon \C_1\to [0,1]$ and  $\widehat{\phi}_2\colon \C_2\to [0,1]$ both onto the unit interval. Since $K\C_i'$ are both countable, $\widehat{\phi}_1\colon \C_1\setminus K\C_1'\to [0,1]$ is a Borel isomorphism onto its image, as is $\widehat{\phi}_2\colon \C_2\setminus K\C_2'\to [0,1]$. Then the composition $\widehat{\phi}_2^{-1}\circ \widehat{\phi}_1\colon \C_1\setminus K\C_1'\to \C_2\setminus K\C_2'$ is a Borel isomorphism onto its image (that also is an order isomorphism).  
\end{proof}
\begin{remark}
In the last proof, we could have restricted the mappings to $\C_i\setminus(K\C_i\cup K\C_i')$ for $i=1,2$. Then the induced map $\widehat{\phi}_2^{-1}\circ \widehat{\phi}_1$ is a homeomorphism as well as an order isomorphism. On the other hand, the mappings we did use map are one-to-one, in particular on the elements of $[0,1]$ that are expressible as fractional representations using the families $\{\Z_{n_i}\mid i > 0\}$ and $\{\Z_{n'_i}\mid i > 0\}$.
\end{remark}
\section{Summary}
We have studied the topological groups structures on the Cantor set $\C$ and shown that any such structure has an ``equivalent" abelian group structure, in the sense that the Haar measures are the same. We also showed any representation of $\C$ as an abelian group admits a continuous mapping onto the unit interval sending Haar measure to Lebesgue measure. Finally, we showed there is a Borel isomorphism between any two group structures on $\C$. 

This work is the outgrowth of a talk by the second author at a Dagstuhl seminar in 2012. A final comment in that talk sketched a domain-theoretic approach to showing that Haar measure on $\C\simeq 2^\Nat$ maps to Lebesgue measure. We were inspired to look more closely at this issue because of the enthusiasm Prakash Panangaden expressed for that result. So, as a 60th birthday present, we offer this paper, and hope the recipient enjoys this presentation as well.
\medbreak 
\emph{Happy Birthday, Prakash!!}

\section*{Acknowledgements}
Both authors gratefully acknowledge the support of the US NSF during the preparation of this paper; the second author also thanks the US AFOSR for its support during the preparation of this paper. 

\bibliographystyle{plain}

\end{document}

%% file: macros.tex
\def\N{{\mathbb N}}    
\def\Z{{\mathbb Z}}    
\def\.{{\cdot}}            
\def\<{\langle}            
\def\>{\rangle}            
\def\({\big(}              
\def\){\big)}              
\def\implies
{\hbox{$\Rightarrow$}}     
\def\lead{\leaders\hbox to 1.5ex{\hss${.}$\hss}\hfill}
\def\arr{\hbox to 60pt{\rightarrowfill}}
\def\larr{\hbox to 60pt{\leftarrowfill}}

\newskip\aline \newskip\halfaline
\def\1{{\bf1}}

\renewcommand{\Re}{{\mathbb R}}
\newcommand{\Nat}{{\mathbb N}}

\newcommand{\C}{{\mathcal{C}}}

\newcommand{\B}{{\mathcal{B}}}

\renewcommand{\H}{{\mathcal{H}}}

\renewcommand{\N}{{\mathcal{N}}}

\newcommand{\ua}{{\uparrow}}
\newcommand{\da}{{\downarrow}}

\renewcommand{\O}{{\mathcal O}}

\newcommand{\Da}{\mathord{\mbox{\makebox[0pt][l]{\raisebox{-.4ex}
                           {$\downarrow$}}$\downarrow$}}}
\newcommand{\Ua}{\mathord{\mbox{\makebox[0pt][l]{\raisebox{.4ex}
                           {$\uparrow$}}$\uparrow$}}}
\renewcommand{\Pr}{\textsf{Prob}}